\documentclass[12pt,reqno]{amsart}

\usepackage{amssymb}

\newtheorem{theorem}{Theorem}
\newtheorem{lemma}[theorem]{Lemma}
\newtheorem{corollary}[theorem]{Corollary}

\begin{document}

\baselineskip=15.5pt

\title[A symplectic analog of the Quot scheme]{A symplectic analog of the Quot scheme}

\author[I. Biswas]{Indranil Biswas}

\address{School of Mathematics, Tata Institute of Fundamental Research,
Homi Bhabha Road, Bombay 400005, India}

\email{indranil@math.tifr.res.in}

\author[A. Dhillon]{Ajneet Dhillon}

\address{Department of Mathematics, Middlesex College, University of
Western Ontario, London, ON N6A 5B7, Canada}

\email{adhill3@uwo.ca}

\author[J. Hurtubise]{Jacques Hurtubise}

\address{Department of Mathematics, McGill University, Burnside
Hall, 805 Sherbrooke St. W., Montreal, Que. H3A 0B9, Canada}

\email{jacques.hurtubise@mcgill.ca}

\author[R. A. Wentworth]{Richard A. Wentworth}

\address{Department of Mathematics, University of Maryland, College Park,
MD 20742, USA}

\email{raw@umd.edu}

\subjclass[2010]{14H60, 14C20}

\keywords{Symplectic Quot scheme, automorphism, symmetric product.}

\date{}

\begin{abstract}
We construct a symplectic analog of the Quot scheme that parametrizes the torsion
quotients of a trivial vector bundle over a compact Riemann surface. Some of its
properties are investigated.\\
\textit{R\'esum\'e.}  \textbf{Une analogue symplectique du sch\'ema Quot.}~\,\,
Nous construisons un analogue symplectique du sch\'ema Quot, qui param\'etrise
les modules quotients à torsion
d'un fibr\'e vectoriel trivial sur une surface de Riemann compacte, et
examinons certaines de ses propri\'et\'es.
\end{abstract}

\maketitle

\section{Introduction}

Let $X$ be a compact connected Riemann surface. For fixed integers $n$ and $d$, let 
${\mathcal Q}({\mathcal O}^{\oplus n}_X, d)$ denote the Quot scheme that parametrizes 
the torsion quotients of ${\mathcal O}^{\oplus n}_X$ of degree $d$ (see \cite{Gr} for 
construction and properties of a general Quot schemes). This particular Quot scheme 
${\mathcal Q}({\mathcal O}^{\oplus n}_X, d)$ arises in the study of moduli space of 
vector bundles of rank $n$ on $X$ \cite{BGL}, \cite{GL}, \cite{bifet}. Being a moduli 
space of vortices it is also studied in mathematical physics (see \cite{Ba} and 
references therein).

Here we consider a symplectic analog of the Quot scheme. Let ${\mathcal O}^{\oplus 
2r}_X$ be the trivial vector bundle on $X$ equipped with a symplectic structure given 
by the standard symplectic form on ${\mathbb C}^{2r}$. Take torsion quotients of it of 
degree $dr$ that are compatible with the symplectic structure (this is explained in 
Section \ref{se2.1}). Fixing $X$, $r$ and $d$, let ${\mathcal Q}$ denote the 
associated symplectic Quot scheme. The projective symplectic group 
$\text{PSp}(2r,{\mathbb C})$ has a natural action on ${\mathcal Q}$. We show that the 
connected component, containing the identity element, of the group of all 
automorphisms of ${\mathcal Q}$ coincides with $\text{PSp}(2r,{\mathbb C})$.

In \cite{BGL}, Bifet, Ghione and Letizia used the usual Quot schemes ${\mathcal 
Q}({\mathcal O}^{\oplus n}_X, d)$ associated to $X$ to compute the cohomologies of 
the moduli spaces of semistable vector bundles of rank $n$ on $X$. Our hope is to be 
able to compute the cohomologies of the moduli space of semistable 
$\text{Sp}(2r,{\mathbb C})$--bundles on $X$ using the symplectic Quot scheme 
${\mathcal Q}$.

\section{The symplectic Quot scheme}

\subsection{Construction of the symplectic Quot scheme}\label{se2.1}

Let
$$
\omega'\, :=\, \sum_{i=1}^r (e^*_i\otimes e^*_{i+r}- e^*_{i+r}\otimes e^*_i)
$$
be the standard symplectic form on ${\mathbb C}^{2r}$. Let $X$ be a compact
connected Riemann surface. The sheaf of holomorphic
functions on $X$ will be denoted by ${\mathcal O}_X$. Let
$$
E_0\,:=\, {\mathcal O}^{\oplus 2r}_X
$$
be the trivial holomorphic vector bundle on $X$ of rank $2r$. The above symplectic
form $\omega'$ defines a symplectic structure on $E_0$, because the fibers of
$E_0$ are identified with ${\mathbb C}^{2r}$. This symplectic structure
on $E_0$ will be denoted by $\omega_0$.

Fix an integer $d\, \geq\, 1$. Let
\begin{equation}
{\mathcal Q}\, :=\, {\mathcal Q}(\omega_0,d)
\end{equation}
be the \textit{symplectic Quot scheme} parametrizing all torsion quotients
$$
\tau_Q\, :\, E_0\, \longrightarrow\, Q
$$
of degree $dr$ satisfying the following condition: there is an effective divisor
$D_Q$ on $X$ of degree $d$ such that the restricted form
$\omega_0\vert_{\text{kernel}(\tau_Q)}$ factors as
\begin{equation}\label{e0}
\text{kernel}(\tau_Q)\otimes \text{kernel}(\tau_Q)\, \stackrel{\omega_0}{\longrightarrow}
\,{\mathcal O}_X(-D_Q)\, \hookrightarrow\, {\mathcal O}_X\, ;
\end{equation}
it should be clarified that $D_Q$ depends on $Q$.
Equivalently, the symplectic Quot scheme $\mathcal Q$ parametrizes all coherent analytic
subsheaves $F\, \subset\, E_0$ of rank $2r$ and degree $-dr$ such that the form
$\omega_0\vert_F$ factors as
$$
F\otimes F\, \stackrel{\omega_0}{\longrightarrow}
\,{\mathcal O}_X(-D)\, \hookrightarrow\, {\mathcal O}_X\, ,
$$
where $D$ is some effective divisor on $X$ of degree $d$ that depends on $F$. This
description of $\mathcal Q$ sends any subsheaf $F$ to the quotient sheaf $E_0/F$.

The above pairing
$$
F\otimes F\, \stackrel{\omega_0}{\longrightarrow}\,{\mathcal O}_X(-D)
$$
produces an injective homomorphism
\begin{equation}\label{mu}
\mu\, :\, F\, \longrightarrow\, {\mathcal O}_X(-D)\otimes F^*
\end{equation}
between coherent analytic sheaves of rank $2r$; the homomorphism
$\mu$ is injective because it is injective over the complement of the support
of $E_0/F$. Since
$$
\text{degree}(F)\,=\, -dr\,=\, \text{degree}({\mathcal O}_X(-D)\otimes F^*)\, ,
$$
the homomorphism $\mu$ in \eqref{mu} is an isomorphism. This means that the pairing
$F_x\otimes F_x\, \stackrel{\omega_0(x)}{\longrightarrow}\, {\mathcal O}_X(-D)_x$ is
nondegenerate for every $x\, \in\, X$. The divisor $D$ is
uniquely determined by $F$ because $\mu$ is an isomorphism. More precisely, consider
the homomorphism of coherent analytic sheaves
$$
\widetilde{\mu}\, :\, F\, \longrightarrow\, F^*
$$
given by the restriction $\omega_0\vert_F$. The divisor $D$ is the scheme theoretic
support of the quotient sheaf $F^*/\widetilde{\mu}(F)$.

The group of all permutations of $\{1\, , \cdots\, ,d\}$ will be denoted by $S_d$.
The quotient
\begin{equation}\label{se}
X^d/S_d
\end{equation}
of $X^d$ for the natural action of
$S_d$ is the symmetric product $\text{Sym}^d(X)$. Let
\begin{equation}\label{e2}
\varphi\, :\, {\mathcal Q}\,\longrightarrow\, \text{Sym}^d(X)
\end{equation}
be the morphism that sends any quotient $Q$ to the divisor $D_Q$ (see \eqref{e0}).

Let
\begin{equation}\label{wtq}
\widetilde{\mathcal Q}\, :=\, {\mathcal Q}({\mathcal O}^{\oplus 2r}_X, rd)
\end{equation}
be the Quot scheme that parametrizes all torsion quotients of ${\mathcal O}^{\oplus 2r}_X
\,=\, E_0$ of degree $rd$. It is an irreducible smooth complex projective variety of
dimension $2r^2d$. For any subsheaf $F$ of $E_0$ with
$E_0/F\, \in\,\widetilde{\mathcal Q}$, we have
\begin{equation}\label{dt}
T_{E_0/F}\widetilde{\mathcal Q}\,=\, H^0(X,\, (E_0/F)\otimes F^*)\, .
\end{equation}
Now assume that $E_0/F\,\in\, \mathcal Q$. First consider the homomorphism 
$F\otimes E_0\, \stackrel{\omega_0}{\longrightarrow}\, {\mathcal O}_X$
obtained by restricting $\omega_0$. Note that
$\omega_0 (F\otimes F)\, \subset\, {\mathcal O}_X(-\varphi(E_0/F))$,
where $\varphi$ is the morphism in \eqref{e2} (see \eqref{e0}). Therefore,
$\omega_0$ produces a homomorphism
$$
\theta\, :\, F\otimes (E_0/F)\, \longrightarrow\, {\mathcal O}_X/({\mathcal O}_X(-
\varphi(E_0/F)))\, .
$$
The subspace
$$
T_{E_0/F}{\mathcal Q}\, \subset\, T_{E_0/F}\widetilde{\mathcal Q}
$$
consists of all homomorphisms $\alpha\, :\, F\, \longrightarrow\, E_0/F$
(see \eqref{dt}) such that
$$
\theta(v\otimes \alpha(w))\,=\, \theta(w\otimes\alpha(v))\, .
$$

Clearly $\mathcal Q$ is a closed subscheme of $\widetilde{\mathcal Q}$.

\begin{lemma}\label{lem1}
The scheme $\mathcal Q$ is an irreducible projective variety of dimension
$d(r^2+r+2)/2$.
\end{lemma}

\begin{proof}
The morphism $\varphi$ in \eqref{e2} is clearly surjective. Also, each
fiber of $\varphi$ is irreducible. Since $\text{Sym}^d(X)$ is irreducible,
we conclude that $\mathcal Q$ is also irreducible.

Take a point $\widehat{x}\, :=\, \{x_1\, \cdots\, ,x_d\}\, \in\, \text{Sym}^d(X)$ such
that all $x_i\, \in\, X$, $1\, \leq\, i\, \leq\, d$, are distinct. Let
\begin{equation}\label{lag}
{\mathbb L}\, \subset\, \text{Gr}({\mathbb C}^{2r}, r)
\end{equation}
be the variety parametrizing all Lagrangian subspaces of ${\mathbb C}^{2r}$
for the standard symplectic form $\omega'$. We have a map
${\mathbb L}^d\,\longrightarrow\, \varphi^{-1}(\widehat{x})$ that sends
any $(V_1\, , \cdots\, , V_d)\, \in\, {\mathbb L}^d$ to the composition
$$
{\mathcal O}_X^{2r} \, \longrightarrow\, {\mathcal O}_X^{2r}\vert_{x_1\cup\cdots
\cup x_d}\,=\, \bigoplus_{i=1}^d {\mathbb C}^{2r}_{x_i} \, \longrightarrow\,
\bigoplus_{i=1}^d {\mathbb C}^{2r}_{x_i}/V_i\, ,
$$
where ${\mathbb C}^{2r}_{x_i}$ is the sheaf supported at $x_i$ with stalk
${\mathbb C}^{2r}$; note that this composition is surjective and hence it
defines an element of $\varphi^{-1}(\widehat{x})$. This map ${\mathbb L}^d\,
\longrightarrow\, \varphi^{-1}(\widehat{x})$ is clearly an isomorphism.
Since $\dim {\mathbb L}\,=\, r(r+1)/2$, it follows that $\dim {\mathcal Q}\,=\,
d(r^2+r+2)/2$.
\end{proof}

\subsection{Vortex equation and stability}

Fix a Hermitian metric $\omega_X$ on $X$; note that $\omega_X$ is K\"ahler.

Take a subsheaf $F\, \subset\, {\mathcal O}^{\oplus 2r}_X$ such that
the quotient ${\mathcal O}^{\oplus 2r}_X/F$ is a torsion sheaf of degree $dr$,
so $F\, \in\, \widetilde{\mathcal Q}$ (see \eqref{wtq}). Let
\begin{equation}\label{f1}
{\mathcal O}^{\oplus 2r}_X\, \hookrightarrow\, F^*
\end{equation}
be the dual of the inclusion of $F$ in ${\mathcal O}^{\oplus 2r}_X$. Note that the
homomorphism in \eqref{f1} defines a $2r$--pair of rank $2r$
\cite[p. 535, (3.1)]{BDW}. Therefore, each element of $\mathcal Q$ defines a $2r$--pair
of rank $2r$.

Take any real number $\tau$, and take an element $z_F\, \in\, \mathcal Q$.
Let $F\, \subset\, {\mathcal O}^{\oplus 2r}_X$ be the subsheaf represented by $z_F$.
We note that the subsheaf $F$ is $\tau$--stable if $rd\, <\, \tau$ (see
\cite[p. 535, Definition 3.3]{BDW} for the definition of $\tau$--stability).

We assume that $\tau\, >\, rd$. Therefore,
all elements of ${\mathcal Q}$ are $\tau$--stable. Hence every $F\, \subset\,
{\mathcal O}^{\oplus 2r}_X$ lying in $\mathcal Q$ admits a unique Hermitian structure
that satisfies the $2r$--$\tau$--vortex equation \cite[p. 536, Theorem 3.5]{BDW}.
 
Take any $F\, \subset\, {\mathcal O}^{\oplus 2r}_X$ represented by an element
of $\mathcal Q$. Let $h_F$ denote the unique Hermitian structure on $F$ that
satisfies the $2r$--$\tau$--vortex equation. The isomorphism $\mu$ in \eqref{mu}
takes $h_F$ to the unique Hermitian structure on ${\mathcal O}_X(-D)\otimes F^*$
that satisfies $2r$--$\tau$--vortex equation for ${\mathcal O}_X(-D)\otimes F^*
\, \in\, {\mathcal Q}$. Indeed, this follows immediately from the uniqueness
of the Hermitian structure satisfying the $2r$--$\tau$--vortex equation.

\section{Automorphisms of the symplectic Quot scheme}

In this section we assume that $\text{genus}(X)\, \geq\, 2$.

The group of all holomorphic automorphisms of $\mathcal Q$ will be denoted by
$\text{Aut}({\mathcal Q})$. Let
$$
\text{Aut}({\mathcal Q})^0\, \subset\, \text{Aut}({\mathcal Q})
$$
be the connected component of it containing the identity element. The group
of linear automorphisms of the symplectic vector space $
({\mathbb C}^{2r}\, , \omega')$ is $\text{Sp}(2r,{\mathbb C})$. The standard
action of $\text{Sp}(2r,{\mathbb C})$ on ${\mathbb C}^{2r}$ produces an action
of $\text{Sp}(2r,{\mathbb C})$ on $\mathcal Q$. The center ${\mathbb Z}/2\mathbb Z$
of $\text{Sp}(2r,{\mathbb C})$ acts trivially on $\mathcal Q$. Therefore,
we get a homomorphism
$$
\rho\, :\, \text{PSp}(2r,{\mathbb C})\,=\,\text{Sp}(2r,{\mathbb C})/
({\mathbb Z}/2{\mathbb Z})\, \longrightarrow\, \text{Aut}({\mathcal Q})^0\, .
$$

\begin{theorem}\label{thm1}
The above homomorphism $\rho$ is an isomorphism.
\end{theorem}

\begin{proof}
The standard action of $\text{Sp}(2r,{\mathbb C})$ on ${\mathbb C}^{2r}$ produces
an action of $\text{PSp}(2r,{\mathbb C})$ on the variety $\mathbb L$ in \eqref{lag}.
This action on $\mathbb L$ is effective. From this it follows
immediately that the homomorphism $\rho$ is injective (recall that the general fiber
of $\varphi$ is ${\mathbb L}^d$).

For each $1\, \leq\, i\, \leq\, d$, let $p_i\, :\,X^d\, \longrightarrow\, X$ be the
projection to the $i$-th factor of the Cartesian product. For each pair
$1\,\leq\, i\, <\, j\, \leq\, d$, let
$$
\Delta_{i,j}\,\subset\, X^d
$$
be the divisor over which the two maps $p_i$ and $p_j$ coincide. Let
$$
\widetilde{U} \,:=\, X^d\setminus (\bigcup_{1\leq i < j\leq d}
\Delta_{i,j})
$$
be the complement. Consider all point $\{x_1\, \cdots\, ,x_d\}\, \in\,
\text{Sym}^d(X)$ such that $x_k\, \not=\, x_\ell$ for all $k\, \not=\,\ell$.
The complement in $\text{Sym}^d(X)$ of the subset defined by such points
will be denoted by $U$. The quotient map
$X^d\, \longrightarrow\, X^d/S_d$ (see \eqref{se}) sends $\widetilde{U}$ to
$U$. The quotient map
\begin{equation}\label{f}
f\, :\, \widetilde{U} \,:=\, X^d\setminus (\bigcup_{1\leq i < j\leq d}
\Delta_{i,j})\, \longrightarrow\, U
\end{equation}
is an \'etale Galois covering with Galois group $S_d$.

The inverse image
$\varphi^{-1}(U)\, \subset\, {\mathcal Q}$ will be denoted by ${\mathcal V}$,
where $\varphi$ is the projection in \eqref{e2}. Let
\begin{equation}\label{vpp}
\varphi'\,:=\, \varphi\vert_{\mathcal V}\, :\, {\mathcal V}\,\longrightarrow\, U
\end{equation}
be the restriction of $\varphi$. We note that ${\mathcal V}$ is a fiber bundle over
$U$ with fibers isomorphic to ${\mathbb L}^d$ (see \eqref{lag} for $\mathbb L$).
In particular, $\mathcal V$ is contained in the smooth locus of ${\mathcal Q}$.

The Lie algebra of $\text{Sp}(2r,{\mathbb C})$ will be denoted by
$\text{sp}(2r,{\mathbb C})$. The Lie algebra of $\text{Aut}({\mathcal Q})^0$ is
contained in the space of algebraic vector fields
$H^0({\mathcal V},\, T{\mathcal V})$ equipped with the Lie bracket operation
of vector fields. Let
\begin{equation}\label{dr}
d\rho\, :\, \text{sp}(2r,{\mathbb C})\, \longrightarrow\,H^0({\mathcal V},\, T{\mathcal V})
\end{equation}
be the homomorphism of Lie algebras associated to the homomorphism $\rho$ of Lie groups. To
prove that $\rho$ is surjective it suffices to show that $d\rho$ is surjective.
We note that $d\rho$ is injective because $\rho$ is injective.

Take any algebraic vector field
$$
\gamma\, \in\, H^0({\mathcal V},\, T{\mathcal V})\, .
$$
Let
$$
d\varphi'\, :\, T{\mathcal V}\,\longrightarrow\, {\varphi'}^*TU
$$
be the differential of the projection $\varphi'$ in \eqref{vpp}. As noted before,
the fibers of $\varphi'$ are isomorphic to ${\mathbb L}^d$, in particular, they
are connected smooth projective varieties, so any holomorphic
function on a fiber of $\varphi'$ is a constant function. This implies
that the section $d\varphi'(\gamma)$ descends to $U$. In other words, 
there is a holomorphic vector field $\gamma'$ on $U$ such that
\begin{equation}\label{a}
d\varphi'(\gamma)\,=\, {\varphi'}^*\gamma'\, .
\end{equation}

Let
\begin{equation}\label{b}
\gamma''\, :=\, f^*\gamma'\, \in\, H^0(\widetilde{U},\, T\widetilde{U})
\end{equation}
be the pullback, where $f$ is the projection in \eqref{f}. Since the vector
field $\gamma$ is algebraic, the above vector field
$\gamma''$ is meromorphic on $X^d$, meaning
$$
\gamma''\, \in\, H^0(X^d,\, (TX^d)\otimes
{\mathcal O}_{X^d}(\sum_{1\leq i < j\leq d} m\cdot \Delta_{i,j}))
$$
for some integer $m$. It is known that there are no
such nonzero sections \cite[Proposition 2.3]{BDH}. Therefore, we have
$\gamma''\, =\, 0$, and hence from \eqref{a} and \eqref{b} it follows that
\begin{equation}\label{dz}
d\varphi'(\gamma)\,=\, 0\, .
\end{equation}

The standard action of $\text{Sp}(2r,{\mathbb C})$ on ${\mathbb C}^{2r}$ produces
an action of $\text{Sp}(2r,{\mathbb C})$ on $\mathbb L$ defined in \eqref{lag}. Let
$$
\text{sp}(2r,{\mathbb C})\,\longrightarrow\, H^0({\mathbb L},\, T{\mathbb L})
$$
be the corresponding homomorphism of Lie algebras. It is known that the above
homomorphism is an isomorphism. Let
$$
T_{\rm rel}\, \longrightarrow\, \widetilde{U}\times_U {\mathcal V}\,
\longrightarrow\, \widetilde U
$$
be the relative algebraic tangent bundle for the projection
$\widetilde{U}\times_U {\mathcal V}\, \longrightarrow\, \widetilde U$. Since,
$$
\widetilde{U}\times_U {\mathcal V}
\,=\, {\widetilde U}\times {\mathbb L}^d\, ,
$$
and $H^0({\widetilde U},\, {\mathcal O}_{\widetilde U})\,=\, \mathbb C$
\cite[Lemma 2.2]{BDH}, we have
\begin{equation}\label{tr}
H^0(\widetilde{U}\times_U {\mathcal V},\, T_{\rm rel})\,=\, H^0({\mathbb L},
\, T{\mathbb L})^{\oplus d}\,=\,\text{sp}(2r,{\mathbb C})^{\oplus d}\, .
\end{equation}
Let
$$
T{\mathcal V}\, \supset\, T_{\varphi'}\, \longrightarrow\, {\mathcal V}
$$
be the relative algebraic tangent bundle for the projection $\varphi'$. Since
$f$ in \eqref{f} is a Galois \'etale covering with Galois group $S_d$, from
\eqref{tr} we have
$$
H^0({\mathcal V},\, T_{\varphi'})\,=\,
H^0(\widetilde{U}\times_U {\mathcal V},\, T_{\rm rel})^{S_d}\,=\,
(\text{sp}(2r,{\mathbb C})^{\oplus d})^{S_d}\,=\,
\text{sp}(2r,{\mathbb C})\, .
$$
Combining this with \eqref{dz} we conclude that
$$
H^0({\mathcal V},\, T{\mathcal V})\,=\,
\text{sp}(2r,{\mathbb C})\, .
$$
In particular the homomorphism $d\rho$ in \eqref{dr} is surjective.
\end{proof}

\begin{corollary}\label{cor1}
The isomorphism class of the variety $\mathcal Q$ uniquely determines the
isomorphism class of $X$, except when $d\,=\, 2\,=\,{\rm genus}(X)$.
\end{corollary}

\begin{proof}
This follows from Theorem \ref{thm1} and \cite{fa}. The argument is similar
to the proof of Theorem 5.1 in \cite{BDH}.
\end{proof}

\end{document}